\documentclass[reqno, 11pt]{amsart}

\usepackage{amsmath}
\usepackage{amsthm}
\usepackage{amsfonts}
\usepackage{amssymb}
\usepackage{url}
\usepackage{enumerate}
\usepackage[pdftex,bookmarks=true]{hyperref}
\usepackage{xcolor}

\renewcommand\P{{\mathbb{P}}}
\newcommand\E{{\mathbb{E}}}

\theoremstyle{plain}
 \newtheorem{theorem}{Theorem}[section]
 \newtheorem{conjecture}[theorem]{Conjecture}
 
 \newtheorem{problem}[theorem]{Problem}

 \newtheorem{lemma}[theorem]{Lemma}

\newtheorem{remark}[theorem]{Remark}

\theoremstyle{definition}

\begin{document}
\title{Packing loose Hamilton Cycles}
\author{Asaf Ferber}
\author{Kyle Luh}
\author{Daniel Montealegre}
\author{Oanh Nguyen}
\maketitle

\begin{abstract}
A subset $C$ of edges in a $k$-uniform hypergraph $H$ is a \emph{loose Hamilton cycle}
if $C$ covers all the vertices of $H$ and there exists a cyclic ordering of these vertices such that the
edges in $C$ are segments of that order and such that every two consecutive
edges share exactly one vertex. The binomial random $k$-uniform
hypergraph $H^k_{n,p}$ has vertex set $[n]$ and an edge set $E$
obtained by adding each $k$-tuple $e\in \binom{[n]}{k}$ to $E$ with
probability $p$, independently at random.

Here we consider the problem of finding edge-disjoint loose Hamilton
cycles covering all but $o(|E|)$ edges, referred to as the
\emph{packing problem}. While it is known that the threshold
probability for the appearance of a loose Hamilton cycle in $H^k_{n,p}$
is $p=\Theta\left(\frac{\log n}{n^{k-1}}\right)$, the best known
bounds for the packing problem are around $p=\text{polylog}(n)/n$. Here we
make substantial progress and prove the following asymptotically
(up to a polylog$(n)$ factor) best possible result: For $p\geq
\log^{C}n/n^{k-1}$, a random $k$-uniform hypergraph $H^k_{n,p}$ with high
probability contains
$N:=(1-o(1))\frac{\binom{n}{k}p}{n/(k-1)}$ edge-disjoint loose
Hamilton cycles.

Our proof utilizes and modifies the idea of ``online sprinkling"
recently introduced by Vu and the first author.
\end{abstract}

\section{Introduction}
For $k \in \mathbb{N}$, the binomial random $k$-uniform hypergraph
$H^k_{n,p}$ consists of a vertex set
$V(H^k_{n,p})=[n]:=\{1,\ldots,n\}$ and an edge set $E(H^k_{n,p})$
which is obtained by adding each $k$-tuple $e \in {[n] \choose
k}$ to $E(H^k_{n,p})$ independently with probability $p$. Note that
for $k=2$, the $H^2_{n,p}$ is the standard binomial random
graph $G_{n,p}$.

The binomial random hypergraph model has introduced many natural
problems analogous to those of the binomial random graph model.
However, in the hypergraph setting, completely new techniques are
often required. In this paper we utilize and modify a new approach
due to Vu and the first author \cite{FerberV2015} to deal with the following
problem (we discuss the precise definition of Hamilton cycles in hypergraphs below).
\begin{problem}Given a hypergraph $H$ with $m$ edges, is it possible
to find a collection of edge-disjoint Hamilton cycles in $H$
which covers all but $o(m)$ edges of $H$?
\end{problem}

The problem of packing Hamilton cycles in random graphs is well
studied and is in fact completely solved (see e.g.
\cite{frieze2005packing}, \cite{knox2015edge},
\cite{krivelevich2012optimal} and the references therein). Moreover, in a
recent paper \cite{ferber2015packing}, an asymptotically optimal
solution to the \emph{directed} random graph model is also given.
Therefore, it is somewhat surprising that for the hypergraph case
so little is known.

For a hypergraph $H=H^k_{n,p}$, there is more flexibility in the
definition of a Hamilton cycle.  Letting $1 \leq \ell \leq k-1$,
a subset $C$ of edges in $H$ is a \emph{type-$\ell$ Hamilton cycle} if $C$ covers all the vertices of $H$ and there exists a cyclic ordering of these vertices such that the edges in $C$ are segments
of that order and such that every two consecutive edges share
exactly $\ell$ vertices. In this work, we study \emph{loose
Hamilton cycles} (or \emph{loose cycles} for brevity) which are type-$\ell$ Hamilton cycles with $\ell =1$.  For a loose cycle, let $e_1, e_2, \dots$ indicate its edges in an ordering induced by the ordering
of the vertices. It follows from the definition that the sets
$e_{i-1}\setminus e_i$ are disjoint sets of size $k-1$ that cover
the entire vertex set. Therefore, $k-1$ divides $n$ is a necessary condition
for the existence of a loose cycle. Extending a result of Frieze \cite{frieze2010loose}, Dudek and Frieze showed in \cite{dudek2011loose} that if $2(k-1)$ divides $n$ and $p = \omega\log n/n^{k-1}$
where $\omega = \omega_n$ is any function tending to infinity with
$n$ then whp \footnote{A sequence of events $(\mathcal E_n)$ is said to occur with
high probability (whp) if $\lim _{n\to\infty} \P(\mathcal E_n)= 1$.
In this paper, all asymptotic notation assumes
that the parameter $n$ tends to infinity.}
there exists a loose cycle. In \cite{dudek2012optimal}, Dudek, Frieze, Loh, and Speiss extended the result to $(k-1)|n$ (for a shorter proof see
\cite{ferber2015closing}). However, the packing problem for
Hamilton cycles seems to be more difficult.  In
the paper of Frieze and Krivelevich \cite{frieze2012packing}, even after restricting the range of $\ell$, some
effort is required to show that an edge-disjoint packing exists for
$p \gg \log^2 n/ n$.  This bound does not address the dependence on
$k$ and is significantly larger than the threshold at which a
Hamilton cycle appears.

In what follows, we close the gap for the case of loose cycles by
showing that up to a polylog term, the above lower bound in $p$ for
the appearance of a loose cycle, also guarantees the existence of an
asymptotically optimal packing. In particular, we prove the
following.

\begin{theorem} \label{main-thm}
Let $k\ge 3$ be an integer.
Assume that $(k-1)|n$ and $p \geq\frac{\log^{2k+2} n}{n^{k-1}}$. Then whp $H^{k}_{n, p}$ contains
$\frac{{n \choose k}p}{n/(k-1)}(1+o(1))$ edge-disjoint
loose Hamilton cycles.
\end{theorem}

\begin{remark}
The number of packed loose cycles $\frac{{n \choose k}p}{n/(k-1)}(1+o(1))$ is optimal in the sense that in expectation, there are ${n\choose k} p$ edges and there are $n/(k-1)$ edges in each loose cycle.
\end{remark}

The proof of the main theorem builds on the idea of ``online
sprinkling" introduced by the first author and Vu in
\cite{FerberV2015}.  Their idea is to embed a desired structure
using a randomized algorithm while simultaneously exposing the
random hypergraph.  This general strategy embeds a particular
structure in every round such that the remaining unexposed parts of
the hypergraph are sufficiently random to iterate the algorithm. In their work, the structure is perfect matchings. For the problem of loose cycles, we institute ``online sprinkling" in a
different way.  In every step of our
algorithm, the use of this sprinkling generates a long loose path
that includes the majority of the vertices.  To complete this
embedding into a cycle, we construct an auxiliary hypergraph (as in
another work of the first author \cite{ferber2015closing}) and apply the known results about the
appearance of a loose cycle.

Our proof technique does not apply to type-$l$ Hamilton cycles for $l$ larger than 1. The reason is that on a subset of roughly $\alpha n$ vertices we expose all the possible $k$-tuples with the ``correct" probability $p=\frac{\text{polylog}(n)}{(\alpha n)^{k-\ell}}$ (that is, the threshold for appearance of a type-$\ell$ cycle) in order to close only one cycle. This procedure is too wasteful and fortunately for $\ell=1$ we only need to pay a $\text{polylog}(n)$ factor, but for larger $\ell$ we get a redundant polynomial of $n$. One way to overcome this problem is to find a way to reuse the edges we exposed in order to close other paths. It would be very interesting to see such a result.
\begin{conjecture}
Let $k>l\ge 2$.
There exists a constant $C$ such that if $(k-l)|n$ and $p \geq\frac{\log^{C}n}{n^{k-l}}$, then whp $H^{k}_{n, p}$ contains
$\frac{{n \choose k}p}{n/(k-l)}(1+o(1))$ edge-disjoint
type-$l$ Hamilton cycles.
\end{conjecture}

\section{Proof of Theorem \ref{main-thm}}

\subsection{Outline}
We begin by working under the assumption that $\frac{\log^{2k+2} n}{n^{k-1}}\le p\le
\frac{2\log^{2k+2} n}{n^{k-1}}$.  We show how to remove this restriction at the end of the proof. Fix $\varepsilon>0$, our goal is to show that whp there exist at least
$(1-2\varepsilon)\frac{{n \choose k}p}{n/(k-1)}$ edge-disjoint loose cycles in $H=H^k_{n,p}$ when $n$ is sufficiently large.
The proof of the result is algorithmic in nature.  We introduce a randomized algorithm in Section \ref{algorithm} that will pack each loose Hamilton cycle one at a time as edges of the hypergraph are exposed.  Our algorithm consists of $N=
(1-\varepsilon)\frac{{n \choose k}p}{n/(k-1)} = O(\text{polylog}(n))$ rounds, each of which
produces one loose Hamilton cycle. Each round is further
divided into several steps, where in each such time step, except the last one, the
algorithm tries to extend a current loose path by one edge. In
particular, the algorithm repeatedly
tries to color certain $k$-tuples with an appropriate edge probability, independently at random, until at least one
$k$-tuple is colored (it is allowed to expose the same $k$-tuple twice). A successful coloring also indicates a successful exposure of an edge in the hypergraph. The coloring is simply a tool to expose edges while labeling the successfully exposed edge with the time step.  A round is terminated by closing the path to create a cycle.  This is done by exposing and identifying a Hamilton cycle in an appropriately defined auxiliary hypergraph.

In Section \ref{propH'}, the analysis will show that whp the algorithm outputs the correct number of cycles and that the cycles are edge-disjoint.  As the exposure of the underlying hypergraph is concurrent with the running of our algorithm and we allow multiple exposures of any $k$-tuple (meaning a $k$-tuple that was not successfully colored can be considered again in the next step), it is also necessary to show that our exposure process generates a hypergraph that can be embedded in $H$. We use a coupling lemma to show that such an embedding exists if the random weight that every $k$-tuple accumulates during the algorithm is bounded by $p$.  We then use a concentration result to guarantee the latter.

\subsection{The algorithm} \label{algorithm}

Let $\omega_n\le \log n$ be a sequence tending to infinity with $n$ to be chosen later.
The following is a randomized algorithm which will generate $(1-2\varepsilon)\frac{{n \choose k}p}{n/(k-1)}$
loose cycles whp for sufficiently large $n$.  We divide our algorithm into $N= (1-\varepsilon)\frac{{n \choose k}p}{n/(k-1)}$ independent rounds. In each round $1\le i\le N$, we perform $K+1$ steps where $K=\Theta(n)$ is a deterministic number to be chosen later. We denote step $j$ in round $i$ by step $(i,j)$ where $1\le j\le K+1$. One successful round will generate a loose Hamilton cycle.

For each round $i=1, \dots, N$, we proceed as follows.

{\bf Generating the Long Path.}

{\bf Step 1.} We randomly assign color ${(i,1)}$ to each $k$-tuple of all the $n\choose k$ tuples independently with probability $n^{-k}$.
We repeat this coloring procedure until at least one $k$-tuple is colored or at most $\omega_n$ times.  If we reach $\omega_n$ and no $k$-tuple has been successfully colored, the entire round fails and we move to the next round. Otherwise, let $T^{i}_1$ be the number of times we perform this coloring procedure, then $T^{i}_1$ is a random variable taking values between 1 and $\omega_n$. At time $T^{i}_{1}$, among the $k$-tuples that are colored, we choose one uniformly at random and call it edge $E^{i}_1$. This is the first edge in our path. Choose an arbitrary ordering on $E^{i}_{1}$, say $E^{i}_{1} = (v^{i}_{1}, \dots, v^{i}_k)$. \\

{\bf Step j ($2\le j\le K$).} At this point, we have already obtained $j-1$ (ordered) edges $E^{i}_s = (v^{i}_{(s-1)(k-1)+1}, \dots, v^{i}_{s(k-1)+1})$ ($1\le s\le j-1$) forming a loose path. Let $R^{i}_{j}$ be the collection of all $k$-tuples whose intersection with the union of the above edges is exactly the last vertex $v^{i}_{(j-1)(k-1)+1}$. Randomly assign color ${(i,j)}$ to the $k$-tuples in $R^{i}_{j}$ independently with probability $q$ where $q$ is a deterministic number to be chosen later ($q$ will be roughly $p$ divided by some polylog$(n)$). We repeat this coloring procedure until at least one $k$-tuple is colored or at most $A_j$ times where $A_j$ is a deterministic number to be chosen later.  If we reach $A_j$ and no $k$-tuple has been successfully colored, the entire round fails and we move to the next round. Otherwise, let $T^{i}_j$ be the number of times we perform this coloring procedure, then $T^{i}_j$ is a random variable taking values between 1 and $A_j$. At time $T^{i}_j$, among the $k$-tuples that are colored, we choose one uniformly at random and call it edge $E^{i}_{j}$. Randomly choose an order on the last $k-1$ vertices of $E^{i}_{j}$, say $E^{i}_{j} = (v^{i}_{(j-1)(k-1)+1}, \dots, v^{i}_{j(k-1)+1})$. \\

{\bf Closing the Path.}

{\bf Step K+1.} We will choose $K$ so that the number of the remaining vertices is
\begin{equation}
 n-K(k-1)-1 = \alpha n \label{K-alpha}
\end{equation}
where $\alpha = o(1)$ some (deterministic) negative power of $\log n$ to be chosen later so that $2(k-1)|(\alpha n+1)$. Let $V^{i}$ be the collection of the remaining vertices. Let $v_0$ be a dummy vertex which will later be appropriately replaced by either the first point $v^{i}_1$ or the last point $v^{i}_{K(k-1)+1}$ of the long path.

Let $R^{i}_{K+1}$ be the collection of all $k$-tuples $e$ in $[n]$ such that $e$ contains at least $k-1$ vertices in $V^{i}$ and $e\setminus V^{i}$, if not empty, is either $\{v^{i}_1\}$ or $\{v^{i}_{K(k-1)+1}\}$. Each $k$-tuple in $R^{i}_{K+1}$ is assigned the color $(i, K+1)$ randomly and independently with probability $r = \omega_n\log n / \alpha^{k-1}n^{k-1}$.

Having colored the $k$-tuples in $R^{i}_{K+1}$, we can generate the edges in the auxiliary $k$-uniform hypergraph $G^{i}:= V^{i}\cup\{v_0\}$. A $k$-tuple $e'$ in $G^{i}$ is colored if one of the following holds
\begin{itemize}
\item $e'\subset V^{i}$ and  $e'$ has color $(i, K+1)$,
\item $e'$ contains $v_0$ and either $\{v^{i}_1\}\cup e'\setminus\{v_0\}$ or $\{v^{i}_{K(k-1)+1}\}\cup e'\setminus\{v_0\}$ (or both) has color $(i, K+1)$.
\end{itemize}

Thus, in the random hypergraph $G^{i}$, each $k$-tuple appears independently with probability at least $r$. We find loose Hamilton cycles in $G^{i}$ with $v_0$ at the intersection of two consecutive edges, one of which inherits the color from an edge in $R^{i}_{K+1}$ that contains $v^{i}_1$ and the other inherits the color from an edge in $R^{i}_{K+1}$ that contains $v^{i}_{K(k-1)+1}$. If there exists at least one such cycle, we choose one at random. Otherwise, we repeat the coloring procedure in $R^{i}_{K+1}$ until at least one such cycle appears or at most $\omega_n$ times.  Let $T^{i}_{K+1}$ be the number of times we perform this coloring procedure. If we have repeated the procedure $\omega_n$ times and found no such cycle, the entire round fails and we move to the next round. Assume that the round does not fail at this step and we obtain a cycle as desired. That is, $v_0$ is contained in two consecutive edges $e'$ and $e''$ such that $E^{i}_{n/(k-1)}:=\{v^{i}_1\}\cup e'\setminus\{v_0\}$ and $E^{i}_{K+1}:=\{v^{i}_{K(k-1)+1}\}\cup e''\setminus\{v_0\}$ have color $(i, K+1)$. Replace the edges $e'$ and $e''$ of the cycle by $E^{i}_{n/(k-1)}$ and $E^{i}_{K+1}$ respectively and insert the long path obtained from step 1 to step $K$ between the two edges to form a loose Hamilton cycle on $[n]$. Return this loose cycle and move to the next round.

\subsection{Properties of the algorithm}\label{propH'}
In this section, we verify some properties of the above  algorithm which will complete the proof of Theorem \ref{main-thm} for $\frac{\log^{2k+2} n}{n^{k-1}}\le p\le
\frac{2\log^{2k+2} n}{n^{k-1}}$. At the end of this section, we will show how to remove this restriction.

\begin{lemma}\label{round-success}
Assume that
\begin{equation}
\sum_{j=2}^{K} \exp\left (-qA_j{{n-(j-1)(k-1)-1}\choose {k-1}}\right ) = o(1) \label{cond-Aj}
\end{equation}

Then whp, there are $N(1+o(1))$ successes among the $N$ rounds.
\end{lemma}

We will later choose $q$ appropriately so that by setting
\begin{equation}\label{set-Aj}
A_j = \frac{2\log n}{q{{n-(j-1)(k-1)-1}\choose {k-1}}} = \Theta\left (\frac{\log n}{q(n-j(k-1))^{k-1}}\right ),
\end{equation}

for all $2\le j\le K$, the $A_j$ are some (positive) power of $\log n$ and condition \eqref{cond-Aj} holds.

\begin{proof}
First, we claim that a round succeeds whp. Indeed, consider round 1 (say); step 1 fails if there is no colored $k$-tuple after $\omega_n$ coloring procedures on all $k$-tuples. This happens with probability $(1-n^{-k})^{\omega_n {n\choose k}}\le \exp\left (-n^{-k} {n\choose k} \omega_n\right) =o(1)$. Similarly, the probability for failure in one of the steps from 2 to $K$ is bounded from above by
\begin{equation}
\sum_{j=2}^{K}(1-q)^{A_j{{n-(j-1)(k-1)-1}\choose {k-1}}} \le \sum_{j=2}^{K} \exp\left (-qA_j{{n-(j-1)(k-1)-1}\choose {k-1}}\right ) = o(1)
\end{equation}
by assumption \eqref{cond-Aj}.

Finally consider step $K+1$. By the aforementioned result in Dudek and Frieze \cite{dudek2011loose}, whp there exists at least one loose cycle in $G^{1}$. By symmetry of vertices, with probability $1/(k-1)+o(1)$ there exists one loose cycle in $G^{1}$ with $v_0$ at the intersection of two consecutive edges. Conditioning on the appearance of such a cycle, with probability at least $1/2$, one of the two consecutive edges inherits the color from an edge containing $v^{1}_{1}$ and the other edge $v^{1}_{K+1}$. In other words, with probability at least $1/2k + o(1)$, there exists a loose cycle with the prescribed properties in the algorithm. Since we repeat the experiment $\omega_n$ times, with probability at least $1 - (1- \frac{1}{2k}+o(1))^{\omega_n} = 1 - o(1)$, there exists one loose cycle with the above property.

In summary, a round succeeds whp. Let $f_n^{-1}$ be an upper bound for the probability of failure in one round where $f_n\to \infty$ with $n$. Then the expected number of failures is $N/f_n$. By Markov's inequality, with probability at least $1 - f_n^{-1/2}$, the number of failures is less than $N/\sqrt{f_n}$. In that event, the number of successful rounds is $N - N/\sqrt{f_n} = N(1-o(1))$.
\end{proof}

\begin{lemma}\label{edge-disj}
The loose Hamilton cycles obtained from the successful rounds are edge-disjoint whp.
\end{lemma}

\begin{proof} By the description of the algorithm, in each round, a $k$-tuple is considered in step 1 and at most once more step between step 2 to step $K+1$. Therefore, the probability that a $k$-tuple is colored in each round is bounded by $p':=\omega_n n^{-k}+\max\{qA_j, r\omega_n\} = O(\text{polylog} (n)/ n^{k-1})$. Since there are
$N=(1-\varepsilon)\frac{\binom{n}{k}p}{n/(k-1)}= O(\text{polylog} (n))$ many rounds, by
applying the union bound we obtain that the probability for having
an edge appearing in two loose cycles is bounded from above by $
\binom{n}{k}{N \choose 2} p'^{2}=o(1)$.
\end{proof}

Let $H'$ be the hypergraph on $[n]$ consisting of all the $k$-tuples that are colored in at least one step of the algorithm. For a fixed $k$-tuple $e\subset [n]$ and $1\le i\le N$, let $R^{i}(e) = \{2\le j\le K+1: e\in R^{i}_{j}\}$. Notice that $R^{i}(e)$ has at most one element. For each $k$-tuple $e$, we define the random weight that $e$ accumulates during the algorithm by
\begin{eqnarray}
Q_e &=& 1 - \prod_{i=1}^{N}\left((1 - n^{-k})^{T^{i}_1}\prod_{j\in R^{i}(e)\cap [2, K] }(1-q)^{T^{i}_{j}}\prod_{j\in R^{i}(e)\cap \{K+1\}} (1-r)^{T^{i}_{K+1}}\right )\nonumber\\
&=& \sum_{i=1}^{N}\left(n^{-k}T^{i}_1 +\sum_{j=2}^{K}qT^{i}_{j}\textbf{1}_{j\in R^{i}(e)} + rT^{i}_{K+1}\textbf{1}_{K+1\in R^{i}(e)}\right)(1+o(1)).\label{Q_e}
\end{eqnarray}

For now, we assume that $Q_e\le p$ holds for all $e$ whp.  Under this assumption we demonstrate that $H'$ can be embedded into $H^{(k)}(n, p)$ whp. Intuitively, no $k$-tuple has accrued so much probability mass that its chance of being an edge is greater than $p$.

\begin{lemma}
	If $Q_e \leq p$ holds for all $e$ whp, then there exists a coupling of $H'$ to $H = H^k_{n,p}$ such that $H' \subset H$ whp.
\end{lemma}

\begin{proof}
	We introduce independent random variables, $U_e$, uniform on $[0,1]$ for each $k$-tuple.  Let $H$ be the random hypergraph in which a $k$-tuple is an edge if $U_e \leq p$.  Observe that $H$ is distributed as $H^k_{n,p}$.

Next, we construct a copy $H''$ of $H'$ such that $H''\subset H$ whp. Note that the algorithm in Section \ref{algorithm} consists of a series of queries, each of which questions whether a certain $k$-tuple will be assigned a color with a particular probability of success. For notational convenience, we enumerate these queries by $1, 2, \dots$ in the order that they are made. To construct $H''$, we will use the $U_e$'s and some independent coin flips to answer these queries. We will recursively define the following partial sums, $\{S_e(t)\}_{t = 0, 1, \dots, }$, which keep track of the query probabilities.

Start the algorithm with $S_e(0) = 0$ for all $k$-tuples $e$. Assume that the algorithm is going to make the $t$-th query and the $S_e(t-1)$ are already defined for all $e$. Assume that the $t$-th query questions whether the $k$-tuple $e_t$ receives a certain color with probability $q_t$ of success. We set $S_e(t)=S_e(t-1)$ for $e\neq e_t$ and set $S_e(t) = S_e(t-1)+q_t(1-S_e(t-1))$ for $e = e_t$. Consider two cases, $U_{e_t}< S_{e_t}(t-1)$ and $U_{e_t}\ge S_{e_t}(t-1)$. In the former, toss a $q_t$-coin independent of all previous random variables to decide the result of the $t$-th query. In the latter, the $t$-th query returns a success if and only if $U_{e_t} \in [S_{e_t}(t-1), S_{e_t}(t)]$. Note that in either case, conditioned on the previous queries, the $t$-th query returns success with probability $q_t$.

A $k$-tuple $e$ is said to be an edge in $H''$ if it is successfully colored in at least one query during the algorithm. Observe that
\begin{itemize}
\item  $H''$ has the same distribution as $H'$,
\item at the last query of the algorithm, $S_e = Q_e$ for all $e$,
\item a $k$-tuple $e$ is colored if and only if $U_e \leq S_e(t)$ for some $t$, or equivalently, $U_e\le Q_e$.
\end{itemize}

Since whp, $Q_e \leq p$ for all $e$, we have $H' \subset H$ whp, proving the desired result.
\end{proof}
To show that whp $Q_e\le p$ for all $e$, we first show that it holds in expectation.
\begin{lemma}\label{expc-qe} If $\alpha = o\left (\frac{1}{\omega_n^{2}\log n}\right )$ and $qn^{k-1} = o(1)$, then $\E (Q_e) \in [p(1-2\varepsilon), p(1-\varepsilon/2)]$ when $n$ is sufficiently large.
\end{lemma}

From now on, we set $q = \frac{1}{n^{k-1}\log n }$ and $\alpha = \frac{1}{\omega_n^{3}\log n}(1+o(1))$ where we choose the $o(1)$ so that $2(k-1)|\alpha n$.

\begin{proof}
Since the $N$ rounds are independent and identically distributed, by linearity of expectation, we have
\begin{eqnarray}
\frac{\E(Q_{e})}{N} &=& \E \left(n^{-k}T^{1}_1 +\sum_{j=2}^{K}qT^{1}_{j}\textbf{1}_{j\in R^{i}(e)} + rT^{1}_{K+1}\textbf{1}_{K+1\in R^{1}(e)}\right)(1+o(1)).\nonumber
\end{eqnarray}

Notice that in each round, the probability a $k$-tuple is not one of the edges of the long path is
$
\frac{{\alpha n \choose k}}{{n \choose k}} =\Theta(\alpha^k)
$. Hence,

\begin{align*}
\frac{\E(Q_{e})}{N} &= O(\omega_n n^{-k}+ \omega_n\alpha^k r)  + \sum_{j=2}^{K} \E(qT^{1}_j) \frac{k}{n}\left(\frac{n - (k-1)j}{n}\right)^{k-1}(1+o(1))  \\
                        &= o(p/N) + \sum_{j=2}^{K} \left(\frac{1}{{n - (k-1)j \choose k-1}} \right)
                        \frac{k}{n} \left(\frac{n - (k-1)j}{n}\right)^{k-1}(1+o(1))  \\
                        &= \frac{p(1-\varepsilon)}{N}(1+o(1))\in \left[\frac{p}{N}(1-2\varepsilon), \frac{p}{N}(1-\varepsilon/2)\right],
\end{align*}

where in the first equality, we use independence of the variables $T^{1}_j$ and $\textbf{1}_{j\in R^{1}(e)}$ and \newline $\P(j\in R^{1}(e)) = \frac{k}{n} \left(\frac{n - (k-1)j}{n}\right)^{k-1}(1+o(1))$ with $\frac{k}{n} $ being the probability that $e$ contains the last vertex of $E^{1}_{j-1}$ and $\left(\frac{n - (k-1)j}{n}\right)^{k-1}$ being the probability that the remaining $k-1$ vertices in $e$ are not covered by any of the $E^{1}_{1}, \dots, E^{1}_{j-1}$. The condition $qn^{k-1} = o(1)$ guarantees that $q|R^{1}_{j}|=o(1)$ and so $\E T^{1}_{j} = \frac{1}{q|R^{1}_{j}|}(1+o(1))$ explaining the second equality.
\end{proof}

Finally, we show that whp $Q_e\leq p$ for all $e$.  We make use of McDiarmid's concentration inequality  \cite{mcdiarmid1998concentration}.
\begin{theorem}\label{concentration}
Let $X_1, \dots , X_t$ be independent random variables, with $a_k \leq X_k \leq b_k$ for each $k$.  Let $S_t = \sum _{k=1}^{t}X_k$ and let $\mu= \E[S_t]$.  Then for each $\lambda \geq 0$,
$$
\P[|S_t - \mu|\geq \lambda] \leq 2 e^{-2 \lambda^2 / \sum (b_k-a_k)^2}.
$$
\end{theorem}
Observe that in the estimate \eqref{Q_e}, $Q_e$ is approximately the sum of $N$ independent random variables, each of which is bounded by $n^{-k}\omega_n  + r \omega_n+ q\max_{2\le j\le K}\{A_j\} \leq 2 \omega_n^{2}\log n/ \alpha^{k-1}n^{k-1}$.
We know that $\E Q_e \leq (1-\varepsilon/2) p$.
Thus, by Theorem \ref{concentration},
\begin{align*}
\P(Q_e> p)\le \P(|Q_e - \E Q_e| &\geq \varepsilon p /3) \leq 2 \exp \left( -\frac{\varepsilon^2 p^2 \alpha^{2k-2}n^{2k-2}}{18 N \omega_n^{4}\log^2 n} \right) \\
                    &\le n^{-\omega_n}
\end{align*}
if we set $\omega_n = \log ^{1/6k} n$. We take a union bound over all $k$-tuples of vertices to yield the claim and complete the proof of the theorem for the case $\frac{\log^{2k+2} n}{n^{k-1}}\le p\le \frac{2\log^{2k+2} n}{n^{k-1}}$.

In the general case when $p$ can be greater than $2\frac{\log^{2k+2} n}{n^{k-1}}$, let $M=\left\lfloor\frac{p n^{k-1}}{\log^{2k+2} n}\right\rfloor$. Let $H = H^{(k)}(n, p)$. Each edge in $H$ is assigned a number from $1$ to $M$ uniformly and independently at random. Let $H_i$ be the graph consisting of edges assigned number $i$. Then $H_i$ has the same distribution as $H^{(k)}(n, p/M)$. We have shown that each $H_i$ contains $\frac{p {n\choose k}}{M n/(k-1)}(1+o(1))$ disjoint loose cycles whp. By the same argument with Markov's inequality as in the proof of Lemma \ref{round-success}, whp there are $M(1+o(1))$ graphs among the $H_i$'s having the aforementioned property. Since the $H_i$ are edge-disjoint, one can add up the number of loose cycles in each $H_i$ and obtain the desired number of loose cycles in $H$.

\section{Acknowledgements}
The authors would like to thank the anonymous referees for their very helpful comments and suggestions.
\bibliographystyle{plain}
\bibliography{ham_loose_4}
\end{document}